\documentclass[12pt]{amsart}
\usepackage{amssymb}
\usepackage{amsmath}
\usepackage{amsfonts}
\usepackage{latexsym}
\usepackage[mathscr]{eucal}
\usepackage{bm}
\usepackage{fullpage}

\numberwithin{equation}{section}

\def ~{\hspace{1mm}}

\newcommand{\clH}{\mathcal{H}}

\def\textmatrix#1&#2\\#3&#4\\{\bigl({#1 \atop #3}\ {#2 \atop #4}\bigr)}
\def\dispmatrix#1&#2\\#3&#4\\{\left({#1 \atop #3}\ {#2 \atop #4}\right)}
\newcommand{\beg}{\begin{equation}}
\newcommand{\eeg}{\end{equation}}
\newcommand{\ben}{\begin{eqnarray*}}
\newcommand{\een}{\end{eqnarray*}}

\newtheorem{thm}{Theorem}[section]
\newtheorem{cor}[thm]{Corollary}
\newtheorem{lem}[thm]{Lemma}

\newtheorem{prop}[thm]{Proposition}

\numberwithin{equation}{section}
\theoremstyle{definition}
\newtheorem{defn}[thm]{Definition}
\newtheorem{eg}[thm]{Example}
\newtheorem{rem}[thm]{Remark}

\def\textmatrix#1&#2\\#3&#4\\{\bigl({#1 \atop #3}\ {#2 \atop #4}\bigr)}
\def\dispmatrix#1&#2\\#3&#4\\{\left({#1 \atop #3}\ {#2 \atop #4}\right)}

\begin{document}
\title[Spectral sets and distinguished varieties in $\Gamma$]
{Spectral sets and distinguished varieties in the symmetrized bidisc}

\author[Sourav Pal]{Sourav Pal}
\address[Sourav Pal]{Department of Mathematics, Ben-Gurion University of the Negev, Be'er Sheva-84105, Israel.}
\email{sourav@math.bgu.ac.il}

\author[Orr Shalit]{Orr Moshe Shalit}
\address[Orr Moshe Shalit]{Department of Mathematics, Ben-Gurion University of the Negev, Be'er Sheva-84105, Israel.}
\email{oshalit@math.bgu.ac.il}

\keywords{Symmetrized bidisc, Distinguished varieties, Spectral
set, Fundamental operator, von-Neumann's inequality, Complete
spectral set}

\subjclass[2010]{47A13, 47A20, 47A25, 47A45}
\thanks{The first author was supported in part by a postdoctoral fellowship funded in part by the Skirball Foundation via the Center for Advanced Studies in Mathematics at Ben-Gurion University of the Negev. The second author is partially supported by ISF Grant no.
474/12, by EU FP7/2007-2013 Grant no. 321749, and by GIF Grant no.
2297-2282.6/20.1.}

\begin{abstract}
We show that for every pair of matrices $(S,P)$, having the closed
symmetrized bidisc $\Gamma$ as a spectral set, there is a one
dimensional complex algebraic variety $\Lambda$ in $\Gamma$ such
that for every matrix valued polynomial $f(z_1,z_2)$,
$$\|f(S,P)\|\leq \max_{(z_1,z_2)\in \Lambda}\|f(z_1,z_2)\|.$$
The variety $\Lambda$ is shown to have the determinantal
representation
$$\Lambda = \{(s,p) \in \Gamma : \det(F + pF^* - sI) = 0\} ,$$
where $F$ is the unique matrix of numerical radius not greater than 1 that
satisfies
$$S-S^*P=(I-P^*P)^{\frac{1}{2}}F(I-P^*P)^{\frac{1}{2}}.$$
When $(S,P)$ is a strict $\Gamma$-contraction, then $\Lambda$ is a
{\em distinguished variety} in the symmetrized bidisc, i.e. a one
dimensional algebraic variety that exits the symmetrized bidisc
through its distinguished boundary. We characterize all
distinguished varieties of the symmetrized bidisc by a
determinantal representation as above.
\end{abstract}

\maketitle

\section{Introduction and notations}

In this paper, we contribute to the understanding of the
relationship between the complex geometry of a domain or a variety
in $\mathbb C^2$ and the properties of commuting operator pairs on
a Hilbert space having that domain or variety as a spectral set.
Additionally, we add to the accumulation of interesting phenomena related to
model theory and dilation theory in finite dimensions \cite{AM05,LS13,MS13}.  A
principal source of inspiration for us is the following sharpening
of Ando's inequality:

\begin{thm}[Agler and M\raise.45ex\hbox{c}Carthy, \cite{AM05}]\label{thm:AMvNineq}
Let $T_1,T_2$ be two commuting contractive matrices, neither of
which has eigenvalue of unit modulus. Then there is a one
dimensional variety $V \subseteq \mathbb D^2$ (so called {\em
distinguished variety}) such that for every polynomial $p$ in two
variables
\[
\|p(T_1,T_2)\| \leq \sup_{(z_1, z_2) \in V}|p(z_1,z_2)| .
\]
\end{thm}
The main aim of this paper is to obtain an analogous result in the
{\em symmetrized bidisc} setting. En route we obtain a
characterization of the distinguished varieties in the symmetrized
bidisc. Interestingly, a significant portion of the proofs and results are
different from the bidisc case. To explain more precisely the
background and main results, we begin by listing a few notations
that will be used in sequel.
\begin{itemize}
\item [$\bullet$] $\mathbb D$, $\mathbb T$ : the open unit disk
and the unit circle of the complex plane; \item [$\bullet$]
$\partial X$ will have two meanings:
\begin{enumerate}
\item for a set $X\subseteq \mathbb C^d$ with non-empty interior,
$\partial X$ denotes the topological boundary of $X$;
\item if $X$
is a variety inside some specified domain $U$, then $\partial X :=
\partial U \cap \overline{X}$;
\end{enumerate}
\item [$\bullet$] $bX\,$: the $\check{\textup{S}}$ilov boundary of
$X$.
\item [$\bullet$] $\mathcal L(\mathcal H)$ : the
algebra of bounded operators on a Hilbert space $\mathcal H$;
\item $l^2(E)$ : the Hilbert space of square summable sequences of
vectors from the Hilbert space $E$;
\item[$\bullet$] $H^2(\mathbb D)$ :
the Hardy space of the disc consisting of holomorphic
functions from $\mathbb D$ to $\mathbb C$ with square summable
coefficients;
\item[$\bullet$] $M_z$ : the Unilateral Shift operator
on $H^2(\mathbb D)$;
\item[$\bullet$] $H^2(E)$ : the vectorial
Hardy space consisting of holomorphic functions from $\mathbb D$
to the Hilbert space $E$ with square summable coefficients;
\item[$\bullet$] $T_{\varphi}$ : Toeplitz operator with symbol
$\varphi$;
\item[$\bullet$] $\sigma(T_1,T_2,\cdots,T_d)$ : the
Taylor joint spectrum of a commuting $d$-tuple of operators
$(T_1,T_2,\cdots,T_d)$;
\item[$\bullet$] $P_{\mathcal H}$:
orthogonal projection onto the space $\mathcal H$;
\item[$\bullet$] $\textup{Ran} T,\;\overline{\textup{Ran}}T$ :
range and range closure of an operator $T$;
\item[$\bullet$]
$D_P=(I-P^*P)^{\frac{1}{2}}$: defect operator of a contraction $P$;
\item[$\bullet$] $\mathcal
D_P=\overline{\textup{Ran}}(I-P^*P)^{\frac{1}{2}}$ : defect space
of a contraction $P$;
\item[$\bullet$] $\omega(T)$ : the
\textit{numerical radius} of an operator $T$ on a Hilbert space
$\mathcal H$ which is defined as $\omega(T) = \sup \{|\langle Tx,x
\rangle|\; : \; \|x\|_{\mathcal{H}} = 1\}.$
\item[$\bullet$] $\pi$ : the
\textit{symmetrization map} defined as $\pi(z_1, z_2) = (z_1+z_2,z_1 z_2)$. 
\end{itemize}

Following the notion introduced by von-Neumann, we say that a
compact set $K\subseteq \mathbb C^d$ is a spectral set for a
$d$-tuple of commuting bounded operators
$\underline{T}=(T_1,T_2,\cdots,T_d)$ defined on a Hilbert space
$\mathcal H$ if $\sigma(\underline{T})\subseteq K$ and the
inequality
$$ \|f(\underline{T})\|\leq \sup_{\underline z \in K}|f(\underline
z)|, \quad \underline z=(z_1,z_2,\cdots,z_d)
$$
holds for every rational function $f$ in $d$-variables with poles
off $K$. Furthermore, $K$ is said to be a complete spectral set
for $\underline{T}$ if for every matrix valued rational function
$f$ in $d$-variables, $$\|f(\underline{T})\|\leq \sup_{\underline
z\in K} \|f(\underline z)\|.
$$
Here $f=[f_{ij}]_{m\times n}$, where each $f_{ij}$ is a scalar valued
rational function in $d$-variables with poles off $K$ and
$f(\underline{T})$ denotes the operator from $\mathcal H^n$ to
$\mathcal H^m$ with block matrix
$[f_{ij}(\underline{T})]_{m\times n}$.

Subsets of $\mathbb C^d$ that are spectral sets or complete
spectral sets for a $d$-tuple of commuting operators, have been
studied for decades and many remarkable results have been
obtained, \cite{BadeaBeckermann,paulsen}. This paper concerns the closed
symmetrized bidisc $\Gamma \subseteq \mathbb C^2$ and a class of
one dimensional algebraic varieties in $\Gamma$ as spectral sets
and complete spectral sets. The set $\Gamma$ and its interior, the
symmetrized bidisc $\mathbb G$, are defined in the following way:
\begin{align*} & \Gamma = \{(z_1+z_2,z_1z_2): \; |z_1|\leq 1, |z_2|\leq 1  \}\subseteq \mathbb C^2
\quad \textup{and}\\& \mathbb G =\{ (z_1+z_2,z_1z_2): \; |z_1|< 1,
|z_2|<1 \}.
\end{align*}
The \textit{distinguished boundary} of the symmetrized bidisc is
denoted by b$\Gamma$ and is defined by
$${b}\Gamma =\{
(z_1+z_2,z_1z_2)\,:\; |z_1|= 1, |z_2|=1  \}\subseteq \Gamma.
$$
It is the $\check{\textup{S}}$ilov boundary of $A(\Gamma)$, the
algebra of continuous complex-valued functions on $\Gamma$ which
are analytic in the interior $\mathbb G$. Clearly, the points of
$\mathbb G$, $\Gamma$ and $b\Gamma$ are the symmetrization of the
points of the bidisc $\mathbb D^2$, the closed bidisc
$\overline{\mathbb D}^2$ and the torus $\mathbb T^2$,
respectively, where the symmetrization map is the following:
$$ \pi: \mathbb C^2 \rightarrow \mathbb C^2, \quad \pi(z_1,z_2) = (z_1+z_2,z_1z_2).$$

Function theory and operator theory related to the set $\Gamma$ have been
studied over past three decades (e.g. \cite{ALY13, ay-jfa, ay-ems, ay-blm, ay-jot, ay-jga,
ay-caot, tirtha-sourav1, tirtha-sourav, Sarkar}).

\begin{defn} A one dimensional algebraic variety set $W\subset \mathbb G$ is said to be a
\emph{distinguished variety in the symmetrized bidisc} if
$\partial W := \overline{W}\cap\partial \mathbb G =\overline{W}\cap b\Gamma$.
\end{defn}
A distinguished variety of the symmetrized bidisc
exits the symmetrized bidisc through its distinguished boundary.

The notion of distinguished variety was introduced by Agler and
M\raise.45ex\hbox{c}Carthy in the bidisc setting in \cite{AM05}. A
{\em distinguished variety in the bidisc} $\mathbb{D}^2$ is an
algebraic variety $V \subset \mathbb{D}^2 $ such that $\partial V
= \overline{V} \cap \partial (\mathbb{D}^2) = \overline{V}\cap
\mathbb{T}^2$. Distinguished varieties
in the bidisc have been investigated further by several researchers (see, e.g., \cite{AM06,JKM12,Knese10}).

In Lemma \ref{lem:DVA}, we show that the points in
a distinguished variety $W$ in $\mathbb G$ are symmetrization of
the points of a distinguished variety $V$ in $\mathbb D^2$ and
vice-versa.

In Theorem 1.12 of \cite{AM05}, Agler and
M\raise.45ex\hbox{c}Carthy gave the following
characterization of distinguished varieties in $\mathbb D^2$:

\begin{thm}[Agler-M\raise.45ex\hbox{c}Carthy, \cite{AM05}]\label{thm:AMDVchar}
A set $V \subset \mathbb{D}^2$ is a distinguished variety in the
bidisc if and only if there is a rational matrix valued inner
function $\psi$ on $\mathbb{D}$ such that
\[
V = \{(z,w) \,\big|\, \det(\psi(z) - wI) =  0\}.
\]
\end{thm}
In Theorem \ref{thm:DVchar}, our first main result of this paper,
we establish the fact that a distinguished variety $W$ in $\mathbb
G$ has the representation
\begin{equation}\label{reprentn}
W = \{(s,p)\in\mathbb G:\,\det (A+pA^*-sI) = 0\} ,
\end{equation}
where $A$ is some matrix with $\omega(A)\leq 1$. Moreover, every subset
$W$ of the above form is a distinguished variety in $\mathbb G$
provided that $\omega(A)<1$. Examples show that a set $W$ of the form
(\ref{reprentn}) with $\omega(A)=1$ may or may not be a
distinguished variety in $\mathbb G$. It is somewhat surprising
that this representation of a distinguished variety in $\mathbb G$
has simpler form than the one in $\mathbb D^2$, described in
Theorem
\ref{thm:AMDVchar}.\\

Let $\underline{T}=(T_1,T_2,\cdots,T_d)$ be a commuting $d$-tuple
of operators such that $\sigma(\underline{T}) \subseteq X$. A {\em
normal $bX$-dilation} of $\underline{T}$ is a commuting $d$-tuple
$\underline{N}=(N_1,\cdots,N_d)$ of normal operators on a larger
Hilbert space $\mathcal K \supseteq \mathcal H$ such that
$\sigma(\underline{N})\subseteq bX$ and
$q(\underline{T})=P_{\mathcal H}q(\underline{N})|_{\mathcal H}$,
for any polynomial $q$ in $d$-variables $z_1,\dots,z_d$. A
celebrated theorem of Arveson states that $\underline T$ has a
normal $bX$-dilation if and only if $X$ is a complete
spectral set of $\underline T$ (Theorem 1.2.2 and its corollary,
\cite{arveson2}). In particular, a necessary condition for $\underline
T$ to have a normal $bX$-dilation is that $X$ be a spectral set
for $\underline T$. A natural question is: {\em when is this condition
sufficient?} In other words, fixing $X \subset \mathbb{C}^d$, one
can ask when does the fact that $X$ is spectral set for
$\underline T$ implies that $\underline T$ has a normal
$bX$-dilation. This question was investigated in several contexts,
and it has been shown to have a positive answer when
$X=\overline{\mathbb D}$ \cite{nagy}, when $X$ is an annulus
\cite{agler-ann}, when $X=\overline{\mathbb D^2}$ \cite{ando} and
when $X=\Gamma$ \cite{ay-jfa, tirtha-sourav}. Also we have failure
of rational dilation on a triply connected domain in $\mathbb C$
\cite{ahr, DM}.

We now define the operator pairs of interest of this paper.

\begin{defn}
A $\Gamma$-\textit{contraction} is pair of commuting operators
$(S,P)$ defined on a Hilbert space $\mathcal H$ for which $\Gamma$
is a spectral set.
\end{defn}
Agler and Young proved that a pair $(S,P)$ is a $\Gamma$-contraction if and only if
$$ \|q(S,P)\|\leq \sup_{z\in \Gamma}
|q(z)|.
$$
for every polynomial $q$ in 2 variables (see \cite{ay-ems}); thus
the definition can be simplified so that it does not involve the
joint spectrum $\sigma(S,P)$ nor rational functions.

It is clear from the definition that if $(S,P)$ is a
$\Gamma$-contraction then so is $(S^*,P^*)$ and $\|S\|\leq 2,
\|P\|\leq 1$. In \cite{ay-jfa}, Agler and Young showed that if
$\Gamma$ is a spectral set for $(S,P)$, then it is a complete
spectral set for $(S,P)$, too. To do that, they introduced the
following operator pencil
\begin{align}
\label{pencil}\rho(S,P)=2(I-P^*P)-(S-S^*P)-(S^*-P^*S),
\end{align}
and proved that $\Gamma$ is a spectral set or complete spectral
set for $(S,P)$ if and only if $\rho(\alpha S, {\alpha}^2P)\geq
0$, for all $\alpha \in\mathbb D$ (Theorem 1.2, \cite{ay-jfa}). It
is instructive to compare the equivalence of conditions
\[
(S,P) \textrm{ is a } \Gamma\textrm{-contraction}\,
\Longleftrightarrow \, \rho(\alpha S, \alpha^2 P) \geq 0,
\]
to the equivalence
\[
T \textrm{ is a } \textrm{contraction}\, \Longleftrightarrow \, I
- TT^* \geq 0.
\]
Following the above analogy, we say that a pair of commuting contractions
 $(S,P)$ is a
\textit{strict $\Gamma$-contraction} if there is a positive number
$c$ such that $\rho(\alpha S,{\alpha}^2P)\geq cI$ for all $\alpha
\in \overline{\mathbb D}$. \\

The same fact that $\Gamma$ is a complete spectral set for a
$\Gamma$-contraction $(S,P)$, has been established by
Bhattacharyya, Pal and Shyam Roy by constructing an explicit
$\Gamma$-isometric dilation of $(S,P)$ (Theorem 4.3 of
\cite{tirtha-sourav}). To construct such a $\Gamma$-isometric
dilation, they introduced the notion of fundamental operator of a
$\Gamma$-contraction. The fundamental operator of a
$\Gamma$-contraction $(S,P)$ is the unique solution of the
operator equation
\begin{align} \label{fundaeqn}
S-S^*P=D_PXD_P, \quad X\in\mathcal L(\mathcal D_P),
\end{align}
where $D_P = (I - P^*P)^{1/2}$ is the defect operator of $P$ and $\mathcal D_P$ the closure of its range. 

Our second main result, Theorem \ref{thm:VN}, is that if
$\Sigma=(S,P)$ is a $\Gamma$-contraction with finite defect index
(i.e., $\dim \mathcal{D}_P < \infty$), such that $P^*$ is pure 
(i.e. $P^n\rightarrow 0$ strongly as $n\rightarrow \infty$), then there is a
one dimensional variety $\Lambda_\Sigma$ in $\mathbb G$, depending
on $\Sigma$, such that $\overline{\Lambda_{\Sigma}}$ is a complete
spectral set for $\Sigma$. The variety $\Lambda_{\Sigma}$ is
precisely given by
$$ \Lambda_{\Sigma}=\{(s,p)\in\Gamma:\; \det (F^*+Fp-sI)=0 \}, $$
where $F$ is the fundamental operator of $(S,P)$. In particular, when $\omega(F)<1$,
we find that $(S,P)$ satisfies a von Neumann type inequality on a distinguished variety.

A corollary of the above result (Corollary \ref{cor:VN}) is that
if $(S,P)$ is a strict $\Gamma$-contraction acting on a finite
dimensional space, then there is a distinguished variety $W
\subset \mathbb G$ such that $\overline{W}$ is a complete spectral
set for $(S,P)$. When $(S,P)$ is not a strict $\Gamma$-contraction
this is no longer true, however there is still a one-dimensional
algebraic subvariety in $\Gamma$ which is a complete spectral set
for $(S,P)$ (see Theorem \ref{last:thm}).

In Section 2, we recall some preliminary results on
$\Gamma$-contractions, along with a few new related results, 
which will be used in subsequent sections.

\section{Operator model in the symmetrized bidisc}

Recall from the introduction that a $\Gamma$-contraction may be defined as follows.
\begin{defn}
A pair of commuting operators $(S,P)$ defined on a Hilbert space
is a $\Gamma$-contraction if for every polynomial $q$ in two complex
variables
$$ \|q(S,P)\|\leq \sup_{(s,p) \in \Gamma}
|q(s,p)| .
$$
\end{defn}

It is easy to write down examples of $\Gamma$-contractions. Indeed, if $T_1,T_2$ are commuting contractions, then their symmetrization
$(T_1+T_2,T_1T_2)$ is a $\Gamma$-contraction.
It is important to note, however, that not all $\Gamma$-contractions arise as symmetrizations of pairs of commuting contractions. If this were so, there would not be much independent interest in studying operator theory on $\Gamma$.

\begin{lem}[\cite{ay-jot}] \label{charsym}
Let $(S,P)$ be a $\Gamma$-contraction. Then
$(S,P)=(T_1+T_2,T_1T_2)$ for a pair of commuting operators
$T_1,T_2$ if and only if $S^2-4P$ has a square root that commutes
with both $S$ and $P$.
\end{lem}

There are special classes of $\Gamma$-contractions like
$\Gamma$-unitaries, $\Gamma$-isometries, $\Gamma$-co-isometries,
etc., in the literature of $\Gamma$-contractions. They are analogous
to unitaries, isometries, co-isometries, etc., in the theory of
single contractions.
\begin{defn}\label{dboundary} A commuting pair $(S,P)$ is called a
$\Gamma$-{\em unitary} if $S$ and $P$ are normal operators and
$\sigma(S,P)$ is contained in the distinguished boundary
$b\Gamma$.
\end{defn}
\begin{defn}
A commuting pair $(S,P)$ is called a $\Gamma$-{\em isometry} if it
the restriction of $\Gamma$-unitary to a joint invariant subspace
of $S$ and $P$.
\end{defn}
\begin{defn}
A $\Gamma$-co-isometry is the adjoint of a $\Gamma$-isometry, i.e.
$(S,P)$ is a $\Gamma$-co-isometry if $(S^*,P^*)$ is a
$\Gamma$-isometry.
\end{defn}
\begin{defn}
A $\Gamma$-contraction $(S,P)$ acting on a Hilbert space $\mathcal
H$ is said to be pure if $P$ is a pure contraction, i.e.
${P^*}^n\rightarrow 0$ strongly as $n\rightarrow \infty$.
Similarly, a $\Gamma$-isometry $(S,P)$ is pure if $P$ is a pure
isometry. 
\end{defn}

\begin{defn}
Let $(S,P)$ be a $\Gamma$-contraction on a Hilbert space $\mathcal
H$. A commuting pair $(T,V)$ defined on $\mathcal K$ is said to be
a $\Gamma$-isometric (or $\Gamma$-unitary) extension if $\mathcal
H\subseteq \mathcal K$, $(T,V)$ is a $\Gamma$-isometry (or a
$\Gamma$-unitary) and $T|_{\mathcal H}=S,\, V|_{\mathcal H}=P$.
\end{defn}

We now present a structure theorem for the class of
$\Gamma$-isometries and a few characterizations along with it. For
an elaborate proof of the following result, see Theorem 2.14 of
\cite{tirtha-sourav}.
\begin{thm} \label{Gamma-isometry}
Let $S,P$ be commuting operators on a Hilbert space $\mathcal{H}.$
The following statements are all equivalent:\begin{enumerate}
\item $(S,P)$ is a $\Gamma$-isometry; \item if $P$ has
Wold-decomposition with respect to the orthogonal decomposition
$\mathcal H=\mathcal H_1\oplus \mathcal H_2$ such that
$P|_{\mathcal H_1}$ is unitary and $P|_{\mathcal H_2}$ is pure
isometry then $\mathcal H_1,\,\mathcal H_2$ reduce $S$ also and
$(S|_{\mathcal H_1},P|_{\mathcal H_1})$ is a $\Gamma$-unitary and
$(S|_{\mathcal H_2},P|_{\mathcal H_2})$ is a pure
$\Gamma$-isometry; \item $(S,P)$ is a $\Gamma$-contraction and $P$
is isometry;
\item
    $P$ is an isometry\;,\;$S=S^*P$ and $r(S)\leq2$; \item
    $r(S)\leq2$ and $\rho(\beta S,{\beta^2P})=0$ for all
    $\beta\in\mathbb{T}$;

Moreover if the spectral radius $r(S)$ of $S$ is less than $2$
then {\em (1)},{\em (2)},{\em (3)} and {\em (4)} are equivalent
to: \item $(2{\beta}P-S)(2-\beta S)^{-1}$ is an isometry, for all
$\beta\in\mathbb{T}.$
\end{enumerate}
\end{thm}

\subsection{The fundamental operator of a $\Gamma$-contraction}

In this subsection, we recall from \cite{tirtha-sourav}, the
notion of the fundamental equation of a pair of commuting
operators $S,P$ with $\|P\|\leq 1$, defined on a Hilbert space
$\mathcal H$. For such a commuting pair, the fundamental equation
is defined in the following way:
\begin{align}\label{fundamental-eqn}
S-S^*P=D_PXD_P, \quad X\in\mathcal L(\mathcal D_P) .
\end{align}
The following result shows that the fundamental equation
(\ref{fundamental-eqn}) has a unique solution when $(S,P)$ is a
$\Gamma$-contraction. We call the unique solution the
\textit{fundamental operator} of the $\Gamma$-contraction.
\begin{thm}\label{fundamentalop}
Let $(S,P)$ be a $\Gamma$-contraction on a Hilbert space $\mathcal
H$. The fundamental equation $S-S^*P=D_PXD_P$ has a unique
solution $F$ in $\mathcal L(\mathcal D_P)$ and $\omega(F)\leq 1$.
\end{thm}
For a proof to this see Theorem 4.2 of \cite{tirtha-sourav}.

\begin{rem}
\em{The fundamental operator of a $\Gamma$-isometry or a
$\Gamma$-unitary $(S,P)$ is the zero operator because $S=S^*P$ in
this case}.
\end{rem}

\begin{prop}\label{end-prop}
Let $(S,P)$ and $(S_1,P_1)$ be two $\Gamma$-contractions on a
Hilbert space $\mathcal H$ and let $F$ and $F_1$ be their
fundamental operators respectively. If $(S,P)$ and $(S_1,P_1)$ are
unitarily equivalent then so are $F$ and $F_1$.
\end{prop}
See Proposition 4.2 of \cite{tirtha-sourav1} for a proof.

The following result is a converse to Theorem \ref{fundamentalop}.
\begin{thm}\label{converse}
Let $\hat{F}$ be an operator defined on a Hilbert space $E$ with
$\omega(\hat{F})\leq 1$. Then there is a $\Gamma$-contraction for
which $\hat{F}$ is the fundamental operator.
\end{thm}
 \begin{proof}
 Let us consider the Hilbert space $H^2(E)$ and the commuting operator pair $(T_{\hat{F}^*+\hat{F}z},T_z)$ acting on it.
 Clearly
 $T_{\hat{F}^*+\hat{F}z}=T_{\hat{F}^*+\hat{F}z}^*T_z$ and $T_z$ is an isometry. Now for $z=e^{2i\theta}\in\mathbb
 T$ we have
 \begin{align*}
 \|\hat{F}^*+\hat{F}e^{2i\theta}\| &=\|e^{-i\theta}\hat{F}^*+e^{i\theta}\hat{F}\|
 \\& =\omega(e^{-i\theta}\hat{F}^*+e^{i\theta}\hat{F}), \textup{ by
 self-adjointness}\\& \leq 2, \textup{ since } \omega(\hat{F})\leq 1.
 \end{align*}
 Therefore by part-(4) of Theorem \ref{Gamma-isometry},
 $(T_{\hat{F}^*+\hat{F}z},T_z)$ is a $\Gamma$-isometry. We now consider the
 $\Gamma$-contraction $(T_{\hat{F}^*+\hat{F}z}^*,T_z^*)$ which is in
 particular a $\Gamma$-co-isometry. We prove that $\hat{F}$ is the
 fundamental operator of $(T_{\hat{F}^*+\hat{F}z}^*,T_z^*)$. Clearly $H^2(E)$ can be identified with the space $l^2(E)$ and
 $(T_{\hat{F}^*+\hat{F}z}^*,T_z^*)$ is unitarily equivalent to the
 multiplication operator pair $(M_{\hat{F}^*+\hat{F}z}^*,M_z^*)$ defined on $l^2(E)$
by
$$ M_{\hat{F}^*+\hat{F}z}=\begin{bmatrix} \hat{F}^*&0&0&\cdots \\
\hat{F} &\hat{F}^*&0&\cdots\\ 0&\hat{F}&\hat{F}^*&\cdots \\
\vdots& \vdots&\vdots&\ddots
\end{bmatrix}, \; M_z=\begin{bmatrix} 0&0&0&\dots\\ I&0&0&\dots\\ 0&I&0&\dots\\ \vdots&\vdots&\vdots&\ddots \end{bmatrix}
\textup{ on } l^2(E),$$ we have
\begin{align*}
& M_{\hat{F}^*+\hat{F}z}^*-M_{\hat{F}^*+\hat{F}z}M_z^*
\\& =\begin{bmatrix} \hat{F}&\hat{F}^*&0&\cdots\\ 0&\hat{F}&\hat{F}^*&\cdots\\
0&0&\hat{F}&\cdots\\ \vdots&\vdots&\vdots&\ddots
\end{bmatrix} - \begin{bmatrix} \hat{F}^*&0&0&\dots\\ \hat{F}&\hat{F}^*&0&\cdots\\ 0&\hat{F}&\hat{F}^*&\cdots\\
\vdots&\vdots&\vdots&\ddots \end{bmatrix}
\begin{bmatrix} 0&I&0&\cdots \\ 0&0&I&\cdots \\ 0&0&0&\cdots \\ \vdots&\vdots&\vdots&\ddots
\end{bmatrix} \\& = \begin{bmatrix} \hat{F}&\hat{F}^*&0&\cdots \\ 0&\hat{F}&\hat{F}^*&\cdots \\
0&0&\hat{F}&\cdots\\ \vdots&\vdots&\vdots&\ddots \end{bmatrix}
- \begin{bmatrix} 0&\hat{F}^*&0&\cdots \\ 0&\hat{F}&\hat{F}^*&\cdots \\
0&0&\hat{F}&\cdots\\ \vdots&\vdots&\vdots&\ddots
\end{bmatrix} \\& = \begin{bmatrix} \hat{F}&0&0&\cdots \\ 0&0&0&\cdots \\ 0&0&0&\cdots \\
\vdots&\vdots&\vdots&\ddots \end{bmatrix}.
\end{align*}
Also
\begin{align*} D_{M_z^*}^2 &=I-M_zM_z^* \\
\\& = \begin{bmatrix} I&0&0&\cdots \\ 0&0&0&\cdots\\
0&0&0&\cdots \\ \vdots&\vdots&\vdots&\ddots
\end{bmatrix} .
\end{align*}
It is clear that $D_{M_z^*}^2=D_{M_z^*}$ and therefore if
$F_1=\begin{bmatrix} \hat{F}&0&0&\cdots \\ 0&0&0&\cdots \\ 0&0&0&\cdots \\
\vdots&\vdots&\vdots&\ddots
\end{bmatrix}$ then $F_1$ is defined on $\mathcal D_{M_z^*}=E\oplus \{0\}\oplus \{0\}\oplus \dots \equiv E$ and
$$M_{\hat{F}^*+\hat{F}z}^*-M_{\hat{F}^*+\hat{F}z}M_z^*=D_{M_z^*}F_1D_{M_z^*}. $$ Therefore
by the uniqueness of fundamental operator, $F_1$ is the
fundamental operator of the $\Gamma$-contraction
$(M_{\hat{F}^*+\hat{F}z}^*,M_z^*)$. Clearly $F_1$ is unitarily
equivalent to $\hat{F}$ on $E$. Now since
$(T_{\hat{F}^*+\hat{F}z}^*,T_z^*)$ on $H^2(E)$ and
$(M_{\hat{F}^*+\hat{F}z}^*,M_z^*)$ on $l^2(E)$ are unitarily
equivalent, so are their fundamental operators by Proposition
\ref{end-prop}. Therefore $\hat{F}$ is the fundamental operator of
$(T_{\hat{F}^*+\hat{F}z}^*,T_z^*)$ on $H^2(E)$.
\end{proof}

\subsection{Model theory for pure $\Gamma$-contractions}

\begin{defn}
 Let $(S,P)$ be a $\Gamma$-contraction on a Hilbert space $\mathcal H$. A pair of commuting
 operators $(T,V)$ defined on a Hilbert space $\mathcal K\supseteq \mathcal H$ is said to be
 a $\Gamma$-\textit{isometric dilation} of $(S,P)$ if $(T,V)$ is a $\Gamma$-isometry and
 $$P_{\mathcal H}(T^mV^n)|_{\mathcal H}=S^mP^n, \quad
 n=0,1,2,\dots .$$ Moreover, the dilation will be called \textit{minimal} if
 \begin{align} \label{idilation} \mathcal K=\overline{\textup{span}}\{ T^mV^nh:\; h\in\mathcal H, m,n=0, 1, 2,
 \cdots\}.
 \end{align}
 A $\Gamma$-\textit{unitary dilation} of $(S,P)$ is defined in a similar way
 by replacing the term $\Gamma$-isometry by $\Gamma$-unitary and
 minimality of such a $\Gamma$-unitary dilation is obtained by varying $m,n$ in (\ref{idilation})
 over all integers.
 \end{defn}

In the dilation theory of a single contraction \cite{nagy}, it
is a notable fact that if $V$ is the minimal isometric dilation of
a contraction $T$ then $V^*$ is a co-isometric extension of $P$.
The (partial) converse is obvious: if $V$ is a co-isometric
extension of $T$ then $V^*$ is an isometric dilation of $T^*$.
Here we shall see that an analogue holds for
$\Gamma$-contractions.

\begin{prop}\label{easyprop3}
Let $(T,V)$ on $\mathcal K\supseteq \mathcal H$ be a
$\Gamma$-isometric dilation of a $\Gamma$-contraction $(S,P)$. If
$(T,V)$ is minimal, then $(T^*,V^*)$ is a $\Gamma$-co-isometric
extension of $(S^*,P^*)$. Conversely, if $(T^*,V^*)$ is a
$\Gamma$-co-isometric extension of $(S^*,P^*)$ then $(T,V)$ is a
$\Gamma$-isometric dilation of $(S,P)$.
\end{prop}
\begin{proof}
We first prove that $SP_{\mathcal H}=P_{\mathcal H}T$ and
$PP_{\mathcal H}=P_{\mathcal H}V$, where $P_{\mathcal H}:\mathcal
K \rightarrow \mathcal H$ is orthogonal projection onto $\mathcal
H$. Clearly
$$\mathcal K=\overline{\textup{span}}\{ T^{m}V^n h\,\,\big|\,\;
h\in\mathcal H \textup{ and }m,n\in \mathbb N \cup \{0\} \}.$$ Now
for $h\in\mathcal H$ we have
\begin{align*} SP_{\mathcal
H}(T^{m}V^n h)=S(S^{m}P^n h)=S^{m+1}P^n h=P_{\mathcal
H}(T^{m+1}V^n h)=P_{\mathcal H}T(T^{m}V^n h).
\end{align*}
Thus we have $SP_{\mathcal H}=P_{\mathcal H}T$ and similarly we
can prove that $PP_{\mathcal H}=P_{\mathcal H}V$. Also for
$h\in\mathcal H$ and $k\in\mathcal K$ we have
\begin{align*} \langle S^*h,k \rangle
=\langle P_{\mathcal H}S^*h,k \rangle=\langle S^*h,P_{\mathcal H}k
\rangle=\langle h,SP_{\mathcal H}k \rangle=\langle h,P_{\mathcal
H}Tk \rangle=\langle T^*h,k \rangle .
\end{align*}
Hence $S^*=T^*|_{\mathcal H}$ and similarly $P^*=V^*|_{\mathcal
H}$. The converse part is obvious.
\end{proof}

A functional model for pure $\Gamma$-contractions was described in
\cite{tirtha-sourav1} (Theorem 3.1). We state this result here
because we shall use this model to prove Theorem \ref{thm:VN}. We
recall from \cite{nagy} the notion of characteristic function of a
contraction. For a contraction $T$ defined on a Hilbert space
$\mathcal H$, let $\Lambda_T$ be the set of all complex numbers
for which the operator $I-zT^*$ is invertible. For $z\in
\Lambda_T$, the characteristic function of $T$ is defined as
\begin{eqnarray}\label{e0} \Theta_T(z)=[-T+zD_{T^*}(I-zT^*)^{-1}D_T]|_{\mathcal D_T}.
\end{eqnarray} By virtue of the relation $TD_T=D_{T^*}T$ (section
I.3 of \cite{nagy}), $\Theta_T(z)$ maps $\mathcal D_T$ into
$\mathcal D_{T^*}$ for every $z$ in $\Lambda_T$.

\begin{thm}\label{modelthm}
Let $(S,P)$ be a pure $\Gamma$-contraction defined on a Hilbert
space $\mathcal H$. Then the operator pair $(I\otimes
{F_*}^*+M_z\otimes {F_*},\, M_z\otimes I)$ defined on $\mathcal
N=H^2(\mathbb D)\otimes \mathcal D_{P^*}$ is the minimal
$\Gamma$-isometric dilation of $(S,P)$. Here $F_*$ is the
fundamental operator of $(S^*,P^*)$, $M_z$ is the multiplication
operator on $H^2(\mathbb D)$. Moreover, $(S,P)$ is unitarily
equivalent to the pair $(S_1,P_1)$ on the Hilbert space $\mathbb
H_P=(H^2(\mathbb D)\otimes \mathcal D_{P^*})\ominus
M_{\Theta_P}(H^2(\mathbb D)\otimes \mathcal D_P)$ defined as
$S_1=P_{\mathbb H_P}(I\otimes {F_*}^*+M_z\otimes {F_*})|_{\mathbb
H_P}$ and $P_1=P_{\mathbb H_P}(M_z\otimes I)|_{\mathbb H_P}.$ Here
$M_{\Theta_P}$ is the multiplication operator from $H^2(\mathbb
D)\otimes \mathcal D_P$ to $H^2(\mathbb D)\otimes \mathcal
D_{P^*}$ corresponding to the multiplier $\Theta_P$, which is the
characteristic function of $P$.
\end{thm}

It is interesting to note that the dilation space for the minimal
$\Gamma$-isometric dilation of $(S,P)$ is no bigger than the
dilation space of the minimal isometric dilation of the
contraction $P$, which is surprising because we are concerned with
a commuting multivariable
dilation (this does not hold in the case of two commuting contraction, see, e.g., \cite[Example 7.12]{DK11}).

In Theorem \ref{Gamma-isometry}, we saw that a
$\Gamma$-isometry can be decomposed into two parts of which one is
a $\Gamma$-unitary and the other is a pure $\Gamma$-isometry. Every $\Gamma$-unitary is a
symmetrization of two commuting unitaries (see Theorem 2.5 of \cite{tirtha-sourav}). Therefore, once we have a model for
pure $\Gamma$-isometries, we have a complete picture of
$\Gamma$-isometries. The following theorem gives a model for pure
$\Gamma$-isometries.

\begin{thm}\label{model1}
Let $(S,P)$ be a commuting pair of operators on a separable
Hilbert space $\mathcal H$. If $(S,P)$ is a pure $\Gamma$-isometry
then there is a unitary operator $U:\mathcal H \rightarrow
H^2(\mathcal D_{P^*})$ such that
$$ S=U^*T_{\varphi}U, \textup{ and } P=U^*T_zU, \textup{ where } \varphi(z)= F_*^*+F_*z, $$
$F_*$ being the fundamental operator of $(S^*,P^*)$. Conversely,
every such pair $(T_{A+A^*z},T_z)$ on $H^2(E)$ for some Hilbert
space $E$ with $\omega(A)\leq 1$ is a pure $\Gamma$-isometry.
\end{thm}
\begin{proof}
The fact that a pure $\Gamma$-isometry $(S,P)$ can be identified
with the pair $(T_{F_*^*+F_*z,T_z})$ on $H^2(\mathcal D_{P^*})$
follows from the model theorem for pure $\Gamma$-contractions,
Theorem \ref{modelthm}.

For the converse, recall that in the course of the proof of
\ref{converse} we showed that if $\omega(A)\leq 1$ then the pair
of Toeplitz operators $(T_{A+A^*z},T_z)$ on $H^2(E)$ is a
$\Gamma$-isometry. Moreover, $(T_{A+A^*z},T_z)$ is a pure
$\Gamma$-isometry as $T_z$ is pure isometry.
\end{proof}

\section{Representation of a distinguished variety in $\mathbb G$}

A distinguished variety $V$ in $\mathbb D^2$ has the following
determinantal representation
\[ V=\{ (z,w):\;\textup{det}(\Psi(z)-wI)=0 \}, \]
for some rational matrix valued inner function $\Psi$ (Theorem
1.12 of \cite{AM05}). The following proposition shows that every
distinguished variety $W$ in $\mathbb G$ can be obtained from a
distinguished variety in $\mathbb D^2$.

\begin{lem}\label{lem:DVA}
Let $W\subseteq \mathbb G$. Then $W$ is a distinguished variety in
$\mathbb G$ if and only if there is a distinguished variety $V$ in
$\mathbb{D}^2$ such that $W=\pi(V)$.
\end{lem}
\begin{proof}
Suppose that $V = \{(z,w) \in \mathbb{D}^2 \,:\, p(z,w) = 0\}$ is
a distinguished variety. Define $\tilde{p}(z,w) = p(z,w) p(w,z)$,
and
$$\tilde{V} = \{(z,w) \in \mathbb{D}^2 \,:\, \tilde{p}(z,w) = 0\}.$$
Then $\tilde{V}$ is also a distinguished variety and
$\pi(\tilde{V}) = \pi(V)$. But since $\tilde{p}$ is a symmetric
polynomial, $\tilde{p}(z,w) = q(z+w,zw) = q\circ \pi (z,w)$ for
some polynomial $q$. Letting
$$
W = \{(s,p) \in \mathbb G \,:\, q(s,p) = 0\}
$$
we have that $W = \pi(\tilde{V})$, hence $\pi(V)$ is a variety. As
$\pi$ maps $\mathbb{T}^2$ (and nothing else) onto $b\Gamma$,
$\pi(V)$ is distinguished.

Conversely, let $W = \{(s,p) \in \mathbb G: q(s,p) = 0\}$ be a
distinguished variety in $\mathbb G$. Then $V = \{(z,w) \in
\mathbb{D}^2 : q \circ \pi(z,w) = 0\}$ is a variety which is
mapped onto $W$, and it must be distinguished.
\end{proof}
So, a distinguished variety in $\mathbb G$ has the following
representation:
$$ W=\{ (z+w,zw):\;\textup{det}(\Psi(z)-wI)=0 \}. $$ But this
is not enough for proving the von-Neumann type inequality because not
all $\Gamma$-contractions $(S,P)$ arise as the symmetrization
$(T_1+T_2,T_1T_2)$ of a pair of commuting contractions $T_1, T_2$ (see Lemma \ref{charsym}).
So, we take initiative to provide another
representation of a distinguished variety in $\mathbb G$ as
a determinantal variety in terms
of the natural coordinates in $\mathbb G$.

If $\mu$ is a positive measure on $\partial W$, we denote by
$H^2(\mu)$ the norm closure of polynomials in $L^2(\partial W,
\mu)$. At this point, our aim is to show that for every
distinguished variety $W$ in $\mathbb G$, there is a finite
regular Borel measure $\mu$ on $\partial W$ such that every
point $w \in W$ gives rise to a bounded evaluation functional on
$H^2(\mu)$. It is convenient to denote
\[
ev_w :f \mapsto f(w) \,\, , \,\, f \in H^2(\mu),
\]
and to denote by $k_w$ the function in $H^2(\mu)$ such that
$\langle f, k_w \rangle  = ev_w(f) = f(w)$
for all $f \in
H^2(\mu)$. We will refer to both $ev_w$ and $k_w$ as evaluation
functionals.

\begin{lem}\label{basiclem2}
Let $W$ be a distinguished variety in $\mathbb G$. There exists a
measure $\mu$ on $\partial W$ such that every point in $W$ gives rise to
 a bounded
point evaluation for $H^2(\mu)$, and such that the span of the bounded
evaluation functionals is dense in $H^2(\mu)$.
\end{lem}
\begin{proof}
Agler and M$^{\textup{c}}$Carthy proved the analogous result for
distinguished varieties
in the bidisc (see \cite[Lemma 1.2]{AM05});
we just push forward their result to the symmetrized bidisc.

Let $W$ be a distinguished variety in the symmetrized bidisc. Let
$V$ be a distinguished variety in the bidisc given by Lemma
\ref{lem:DVA}, such that $\pi(V) = W$. By \cite[Lemma 1.2]{AM05}
there is a (finite regular Borel) measure $\nu$ on $\partial V$
such that every point $v \in V$ is a bounded point evaluation for
$H^2(\nu)$, and such that the span of these functionals is dense
in $H^2(\nu)$. Let $\mu$ be the push forward $\mu = \pi_* \nu$,
defined by
\[
\mu(E) = \nu(\pi^{-1}(E)) \quad, \quad E \textrm{ a Borel subset of } \partial W.
\]
Define $U: H^2(\mu) \rightarrow H^2(\nu)$ by first declaring $U f
= f \circ \pi$ for every polynomial $f$. By definition of $\mu =
\pi_* \nu$, $U$ preserves the norm, hence extends to a an isometry
on all of $H^2(\mu)$. If $w \in W$ and $v \in V$ satisfy $\pi(v) =
w$, then for every polynomial $f$
\[
f(w) = f(\pi(v)) = Uf(v),
\]
hence $ev_w = ev_v \circ U$. This shows that every $w \in W$ gives
rise to a bounded point evaluation on $H^2(\mu)$.

Next, we compute that for all $v \in V$ and all $f \in H^2(\mu)$
\[
\langle U^* k_v, f \rangle = \langle k_v, f \circ \pi \rangle = f(\pi(v)),
\]
hence $U^* k_v = k_{\pi(v)}$. But $U$ is an isometry, thus $U^*$
is surjective. Since point evaluations are dense in $H^2(\nu)$, it
follows that point evaluations are dense in $H^2(\mu)$ too.
\end{proof}

By the previous lemma, $H^2(\mu)$ is a reproducing kernel Hilbert space on $W$. The following lemma gives additional information.

\begin{lem}\label{lemeval}
Let $W$ be a distinguished variety in  $\mathbb G$, and let $\mu$
be the measure on $\partial W$ given as in Lemma
\textup{\ref{basiclem2}}. A point $(s_0,p_0) \in \mathbb G$ is in $W$ if and
only if $(\bar s_0, \bar p_0)$ is a joint eigenvalue for
$M_s^*$ and $M_p^*$. 
\end{lem}
\begin{proof}
It is a well known fact in the theory of reproducing kernel Hilbert spaces that $M_f^* k_x = \overline{f(x)} k_x$ for every multiplier $f$ and every kernel function $k_x$; in particular every point $(s_0,p_0) \in W$ is a joint eigenvalue for $M_s^*$ and $M_p^*$. 

Conversely, if $(\bar s_0, \bar p_0)$ is a joint eigenvalue and $v$ is a unit eigenvector, then $f(s_0,p_0) = \langle  v, M_f^* v\rangle$ for every polynomial $f$. Therefore $|f(s_0,p_0)| \leq \|M_f\| = \sup_{(s,p) \in W} |f(s,p)|$. So $(s_0,p_0)$ is in the polynomial convex hull of $W$ (relative to $\mathbb G$), which is $W$. 
\end{proof}

\begin{lem}\label{lempure}
Let $W$ be a distinguished variety in  $\mathbb G$, and let $\mu$
be the measure on $\partial W$ given as in Lemma
\textup{\ref{basiclem2}}. The pair $(M_s,M_p)$ on $H^2(\mu)$,
defined as multiplication by the co-ordinate functions, is a pure
$\Gamma$-isometry.
\end{lem}
\begin{proof}
Let us consider the pair of operators
$(\tilde{M}_s,\tilde{M}_p)$, multiplication by co-ordinate functions,
on $L^2(\partial W, \mu)$. They are commuting normal operators and
the joint spectrum $\sigma(\tilde{M}_s,\tilde{M}_p)$ is contained in 
$\partial W \subseteq b\Gamma$. Therefore, $(\tilde{M}_s,\tilde{M}_p)$ is a
$\Gamma$-unitary and $(M_s,M_p)$, being the restriction of
$(\tilde{M}_s,\tilde{M}_p)$ to the common invariant subspace
$H^2(\mu)$, is a $\Gamma$-isometry. By a standard computation, for every
$(s_0, p_0) \in W$, the kernel function
$k_{(s_0, p_0)}$ is an eigenfunction of $M_p^*$ corresponding to the
eigenvalue $\overline{p_0}$. Therefore, $$ (M_p^*)^nk_{(s_0, p_0)}=\overline{
p_0}^nk_{\lambda} \rightarrow 0 \; \textup{ as } n \rightarrow
\infty,
$$
because $|p_0|<1$.
Since the evaluation
functionals $k_{\lambda}$ are dense in $H^2(\mu)$, this shows that $M_p$ is pure. Hence $M_p$ is a pure
isometry and consequently $(M_s,M_p)$ is a pure $\Gamma$-isometry
on $H^2(\mu)$.
\end{proof}

The following theorem gives a determinantal representation of
distinguished varieties in the symmetrized bidisc in terms of
the natural coordinates in $\mathbb G$.

\begin{thm}\label{thm:DVchar}
Let $A$ be a square matrix A with $\omega(A)< 1$, and let $W$ be
the subset of $\mathbb G$ defined by
\begin{align}\label{eq:W}
W = \{(s,p) \in \mathbb G \,:\, \det(A + pA^* - sI) = 0\}.
\end{align}
Then $W$ is a distinguished variety. Conversely, every
distinguished variety in $\mathbb G$ has the form $\{(s,p) \in
\mathbb G \,:\, \det(A + pA^* - sI) = 0\}$, for some matrix $A$
with $\omega(A)\leq 1$.
\end{thm}

\begin{proof}
Suppose that
\begin{align}
W = \{(s,p) \in \mathbb G \,:\, \det(A + pA^* - sI) = 0\},
\end{align}
where $A$ is an $n \times n$ matrix with $\omega(A) < 1$, and
suppose that $(s,p) \in W$. In order to prove that $W$ is a
distinguished variety, it suffices to show that it exits the
symmetrized bidisc through its distinguished boundary $b\Gamma$,
i.e. if $(s,p)\in \overline{W}\cap \partial \mathbb G=\partial W$
then $(s,p)\in \overline{W}\cap b\Gamma$. This is same as saying
that if $(s,p)=(z_1+z_2,z_1z_2)\in\overline W$ and $|z_1|<1$, then
$|z_2|<1$ as well.

Assume, therefore, that $|z_1|<1$. Since $(z_1 + z_2,z_1 z_2) \in
\overline W$, we have that $\det(A + A^*z_1z_2 - (z_1+z_2)I) = 0$.
It follows that there exists a unit vector $v \in \mathbb{C}^n$
such that
$$
Av + z_1 z_2 A^* v - (z_1 + z_2) v = 0.
$$
Taking the inner product of this equation with $v$, and putting
$\alpha = \langle Av, v \rangle$, we obtain
\begin{align}\label{eq:alpha}
\alpha + z_1 z_2 \bar{\alpha} - z_1 - z_2 = 0. \end{align} Since
$\omega(A) < 1$ it follows that $|\alpha|< 1$. Rearrange
(\ref{eq:alpha}) as
\begin{align}\label{eq:alpha1}
\frac{z_1-\alpha}{1- \bar{\alpha} z_1} = - z_2 .
\end{align}
As $|\alpha|<1$, the left hand side is an automorphism of the
disc. Therefore, if $|z_1|<1$, then $|z_2|<1$. The proof of the
first part of the theorem is complete.

Next we assume that $W$ is a distinguished variety in $\mathbb G$.
Let $(M_s,M_p)$ be the pair of operators on $H^2(\mu)$ given by
multiplication by the co-ordinate functions, where $\mu$ is as in
Lemma \ref{basiclem2}. Then, by Lemma
\ref{lempure}, $(M_s,M_p)$ is a pure $\Gamma$-isometry on
$H^2(\mu)$.
Now $M_pM_p^*$ is projection onto $\textup{Ran} M_p$, and clearly
$$\textup{Ran} M_p \supseteq \{pf(s,p)\,\,:\,\; f(s,p) \textup{ is polynomial in } s,p\}. $$
Let $W = \{(s,p) \in \mathbb G : Q(s,p) = 0\}$, where $Q$ is an appropriate polynomial.
Since $W$ is distinguished, $Q$ is not divisible by $p$. Write
$$Q(s,p) = a_0 + a_1 s + \dots + a_k s^k + pR(s,p) ,$$
where $a_k \neq 0$ and $R$ is some polynomial. Then
$$ s^k \in \textup{Ran} M_p + {\textrm{span}} \{1,s,s^2,\cdots,s^{k-1} \}, $$
and by induction we have that
$$ H^2(\mu)=\textup{Ran} M_p + {\textrm{span}} \{1,s,s^2,\cdots,s^{k-1} \},$$
where if $k = 0$ then we interpret the second summand above as the zero subspace.
Therefore, $\textup{Ran}(I-M_pM_p^*)$ has finite dimension, say
$n$. So by Theorem \ref{model1}, $(M_s,M_p)$ can be identified
with the pair of Toeplitz operators $(T_{\varphi},T_z)$ on the space
$H^2(\mathcal D_{M_p^*})$, where $\mathcal D_{M_{p}^*}$ has
dimension $n$. Here $\varphi(z)=B^*+zB$, where $B$ is the
fundamental operator of $(M_s^*,M_p^*)$.

By Lemma \ref{lemeval}, a point $(s_0,p_0) \in \mathbb G$ is in $W$ if and
only if $(\bar s_0, \bar p_0)$ is a joint eigenvalue for
$M_s^*$ and $M_p^*$. 
In terms of the unitarily equivalent model for pure
$\Gamma$-isometries (Theorem \ref{model1}), this is equivalent to $(\bar s_0, \bar p_0)$ being a joint
eigenvalue of $(T_{\varphi}^*, T_z^*)$, which happens if and only if $\bar s_0$ is
and eigenvalue for $\varphi(p_0)^*$.
This leads to
$$W= \{(s,p) \in \mathbb G \,:\, \det(F_*^* + pF_* - sI) =
0\},$$ where $F_*$, being the fundamental operator of a
$\Gamma$-contraction, has numerical radius not bigger than $1$.
This gives (\ref{eq:W}) with $A = F_*^*$.
\end{proof}

Theorem \ref{thm:DVchar} leaves open the question whether every
variety given by the determinantal representation (\ref{eq:W}),
where the matrix  $A$ satisfies $\omega(A)=1$, is a distinguished variety in the symmetrized bidisc. The following
examples show that the answer to this question is sometimes yes
and sometimes no.

\begin{eg}
Let
\[
A = \begin{pmatrix}
0 & 2 & 0 \\
0 & 0 & 0 \\
0 & 0 & 0
\end{pmatrix} .
\]
Then $\omega(A) = 1$. Define
\[
W = \{(s,p) \in \mathbb G \,:\, \det(A + A^*p - sI) = 0\} .
\]
Computing the determinant, we find that
\[
W = \{(s,p) \in \mathbb G \,:\, s(s^2-4p) = 0\}.
\]
Evidently, this is a distinguished variety.
\end{eg}

\begin{eg}
Let
\[
A = \begin{pmatrix}
0 & 2 & 0 \\
0 & 0 & 0 \\
0 & 0 & 1
\end{pmatrix} .
\]
Then $\omega(A) = 1$. Define
\[
W = \{(s,p) \in \mathbb G \,:\, \det(A + A^*p - sI) = 0\} .
\]
Computing the determinant, we find that
\[
W = \{(s,p) \in \mathbb G \,:\, (1+p-s)(s^2-4p) = 0\}.
\]
This is not a distinguished variety; for example $\overline{W}$
contains the point $(1,0) = \pi(1,0)$, which lies on $\partial
\mathbb G \setminus b \Gamma$.
\end{eg}

In fact, if $A$ has an eigenvalue of modulus $1$ then $W$
defined as in (\ref{eq:W}) is not distinguished. Indeed, suppose
that $A$ has an eigenvalue $\alpha$ of unit modulus, and let $v$
be a corresponding eigenvector. Let $W$ be as above. Then $Av -
\alpha I v=0$, which means that $(\alpha, 0) =  \pi(\alpha,0)$ is
in $\partial \mathbb G \cap \overline{W} \setminus b \Gamma$. We
do not know if for $A$ satisfying $\omega(A) = 1$, having an
eigenvalue of unit modulus is the only obstruction to being
distinguished.

The varieties of the form (\ref{eq:W}) with $\omega(A)<1$ do have
a complete characterization. Denote
\[
b D_{\Gamma} = \{\pi(z,z)=(2z,z^2) \,:\, z \in \mathbb T\}.
\]

\begin{thm}
Let $W$ be a variety in $\mathbb G$. Then
\[
W = \{(s,p) \in \mathbb G \,:\, \det(A + A^*p - sI) = 0\}
\]
for a matrix $A$ with $\omega(A)<1$ if and only if $W$ is a
distinguished variety such that $\partial W \cap bD_{\Gamma} =
\emptyset$.
\end{thm}

\begin{proof}
By Theorem \ref{thm:DVchar}, if $W$ has such a determinantal
representation then $W$ is a distinguished variety. We need to
show that $\partial W \cap bD_{\Gamma} = \emptyset$. For this,
note that
\[
b\Gamma = \{(s,p) \in \Gamma \,:\, |p|=1\}.
\]
and
\[
bD_{\Gamma} = \{(s,p) \in \Gamma \,:\, |p|=1, |s| = 2\}.
\]
If $(s, e^{i\theta}) \in \partial W$, then $\det (A +
A^*e^{i\theta} - sI) = 0$. In other words, $e^{-i\theta/2}s$ is an
eigenvalue of the normal matrix $N = e^{-i\theta/2} A +
e^{i\theta/2}A^*$. But $\|N\| = \omega(N) <2$, thus $|s|<2$, so
$(s,e^{i\theta}) \notin bD_\Gamma$.

Conversely, suppose that $W$ is a distinguished variety such that
$\partial W \cap bD_{\Gamma} = \emptyset$. As in the proof of
Theorem \ref{thm:DVchar}, construct the measure $\mu$ on $\partial
W$ and the space $H^2(\mu)$. Let $(M_s,M_p)$ be the multiplication
operators on $H^2(\mu)$. In the course of the proof of Theorem
\ref{thm:DVchar} we showed that $W$ is given by (\ref{eq:W}) with
$A$ the adjoint of the fundamental operator of $(M_s^*,M_p^*)$.
What we need to show is that $\omega(A)<1$.

In the proof of Theorem \ref{thm:DVchar} we also noted that
$M_s^*$ is unitarily equivalent to $T^*_\varphi$, where $\varphi(z) = A +
A^*z$. Now $\|M_s^*\| = \|M_s\| = \displaystyle \sup_{(s,p) \in
\partial W}|s| <2$ by assumption. On the other hand, $\|M_s^*\| =
\|T_\varphi\| = \displaystyle \max_{z\in \mathbb T}\|\varphi(z)\|$, thus
\begin{equation}\label{eq:omega_small}
2>\|M_s^*\|=\max_{\theta \in \mathbb R}\|e^{i\theta}A +
e^{-i\theta}A^*\| = \omega(e^{i\theta}A + e^{-i\theta}A^*).
\end{equation}
This implies that $\omega(A) < 1$. Indeed, if $\omega(A) = 1$ then
there is some $\theta \in \mathbb R$ and a unit vector $u$ such
that $\langle Au,u \rangle = e^{-i\theta}$. But then
\[
\langle (e^{i\theta}A + e^{-i\theta}A^*)u, u \rangle = 2,
\]
and this contradicts (\ref{eq:omega_small}). Hence the proof is
complete.
\end{proof}


\section{A von-Neumann type inequality for $\Gamma$-contractions}

\subsection{Strict $\Gamma$-contractions and their fundamental operators}

Recall from the introduction that we denote
\[
\rho(S,P)=2(I-P^*P)-(S-S^*P)-(S^*-P^*S).
\]

\begin{defn}
A pair of commuting operators $(S,P)$ is said to be a {\em strict 
$\Gamma$-contraction} if there is a constant $c>0$ such that $
\rho(\alpha S,\alpha^2 P) \geq cI $ for all $\alpha \in
\overline{\mathbb{D}}$.
\end{defn}

\begin{lem}\label{lem:strict}
If $(S,P)$ is a strict $\Gamma$-contraction then $P$ is a strict
contraction. In particular, $P$ and $P^*$ are pure contractions.
\end{lem}
\begin{proof}
Since $
\rho(\alpha S,\alpha^2 P) \geq cI $ for all $\alpha \in
\overline{\mathbb{D}}$, we have in particular
\[
4D_P^2 = \rho(-S,P) + \rho(S,P) \geq 2cI,
\]
thus $D_P^2 \geq (c/2) I$. This in turn implies $P^*P \leq
(1-c/2)I$, so $\|P\|<1$, as required.
\end{proof}
It is easy to see that $P$ is a strict contraction if and only if
$D_P$ is invertible. Thus, a strict $\Gamma$-contraction has
finite defect index if and only if it acts on a finite dimensional
space. Note also that if $D_P$ is invertible, then the fundamental
operator (recall equation (\ref{fundamental-eqn})) of $(S,P)$ is
given by
\[
F = D_P^{-1}(S-S^*P)D_P^{-1}.
\]

\begin{prop}\label{prop:nrstrict}
If $F$ is the fundamental operator of a strict
$\Gamma$-contraction, then $\omega(F)< 1$.
\end{prop}
\begin{proof}
It is well known (see, e.g., \cite[Lemma 2.9]{tirtha-sourav} for additional details) that for every operator $X$,
\[
\omega(X) \leq 1 \Longleftrightarrow \forall \alpha \in \mathbb{T} \,,\, \textup {Re} (\alpha X) \leq I.
\]
It follows immediately that for every operator $X$ and all $c>0$,
\begin{equation}\label{eq:nr}
\omega(X) \leq c \Longleftrightarrow \forall \alpha \in \mathbb{T} \,,\, \textup {Re} (\alpha X) \leq cI.
\end{equation}
Now assume that $(S,P)$ is a strict $\Gamma$-contraction, so
\[
\rho(\alpha S,\alpha^2 P) = 2 D_P^2 -2 \textup{Re}(\alpha(S-S^*P)) \geq cI
\]
for all $\alpha \in \overline{\mathbb T}$. Rearranging, we find that
\[
\textrm{Re}(\alpha(D_P^{-1}(S-S^*P)D_P^{-1}) \leq I - c/2 D_P^{-2} \leq (1-c/2)I,
\]
for all $\alpha \in \mathbb T$.
But the fundamental operator for $(S,P)$ is $F = D_P^{-1}(S-S^*P)D_P^{-1}$. It follows from (\ref{eq:nr}) that $\omega(F)\leq (1-c/2)<1$.
\end{proof}

\begin{rem}
The class of strict $\Gamma$-contractions is analogous to the
class of strict contractions in one variable operator theory. It
is actually a subclass of the $\Gamma$-contractions having
fundamental operators of numerical radius less than 1 (obvious
from the previous result). Therefore, keeping in mind the
importance of the class of strict $\Gamma$-contractions, we make
the statement of the Corollary \ref{cor:VN} precise for strict
$\Gamma$-contractions although it is valid for the larger class of
$\Gamma$-contractions that have fundamental operators of numerical
radius less than 1.
\end{rem}

\subsection{A von-Neumann type inequality}

\begin{thm}\label{thm:VN}
Let $\Sigma=(S,P)$ be a $\Gamma$-contraction on a Hilbert space
$\mathcal H$ such that $(S^*,P^*)$ is pure, and suppose that $\dim \mathcal
D_P < \infty$. Denote by $F$ the fundamental operator of
$(S,P)$, and let
$$
\Lambda_{\Sigma} = \{(s,p) \in \Gamma : \det(F^* + pF - sI) =
0\}.
$$
Then for every matrix valued polynomial $f$,
$$
\|f(S,P)\| \leq \max_{(s,p) \in \Lambda_{\Sigma}\cap b\Gamma}
\|f(s,p)\|.
$$
Moreover, when $\omega(F)<1$, $\Lambda_{\Sigma}\cap \mathbb G$ is a
distinguished variety in $\mathbb G$.
\end{thm}

\begin{proof}
By applying the model theorem for pure $\Gamma$-contractions
(Theorem \ref{modelthm}) to $(S^*, P^*)$, we have that $\mathcal H
\subseteq H^2(\mathbb{D}) \otimes \mathcal{D_P}$, where $\mathcal
D_P$ is finite dimensional --- say of dimension $n$ --- and that
$$
S = (I \otimes F^* + M_z \otimes F)^* \big|_{\clH} \textup{ and }P
= (M_z \otimes I)^* \big|_{\clH}\,,
$$
where $F$ is the fundamental operator of $(S,P)$. Let $\varphi$
denote the $\mathcal{L}(\mathcal{D}_P)$-valued function $\varphi(z) = F^* + z F$.
Let $f$ be a matrix valued polynomial and let $f^\cup$ be the
polynomial satisfying $f^\cup(A,B) = f(A^*,B^*)^*$ for all
operators $A$ and $B$. Then
\begin{align*}
\|f(S,P)\| &\leq \|f(M_\varphi^*, M_z^*)\|_{H^2 \otimes \mathcal D_P} \\
&=  \|f^\cup(M_\varphi, M_z)\|_{H^2 \otimes \mathcal D_P} \leq \|f^\cup(M_\varphi,M_z)\|_{L^2 \otimes \mathcal D_P} \\
& = \max_{\theta \in [0,2\pi]} \|f^\cup (\varphi(e^{i\theta}),
e^{i\theta})\| .
\end{align*}
Now for $e^{i\theta} \in \mathbb{T}$, the matrix
$\varphi(e^{i\theta})$ is normal, because $\varphi(e^{i\theta}) =
e^{i\theta/2}(e^{-i\theta/2} F^* + e^{i\theta/2}F)$. Thus
\begin{align*}
\|f^\cup(\varphi(e^{i\theta}),e^{i\theta}I)\| &=
\max \{f^\cup (\lambda, e^{i\theta}) : \lambda \in \sigma(\varphi(e^{i\theta}))\} \\
&= \max \{f^\cup (\lambda, e^{i\theta}) : \det (F + e^{i\theta}F^*
- \lambda I) = 0 \}.
\end{align*}
Define
$$\Lambda_\Sigma = \{(s,p) \in \Gamma : \det(F^* + pF - sI) =
0\}$$
and
\begin{align*}
\Lambda^*_{\Sigma}
&=  \{(s,p) \in \Gamma : \det(F + pF^* - sI) =
0\} \\
&= \{(\overline{s},\overline{p}) : (s,p)
\in \Lambda_{\Sigma}\} .
\end{align*}
We now want to show that if $\det(F + e^{i\theta}F^* - \lambda I) = 0$ then $(\lambda,e^{i\theta}) \in \Gamma$ (and thus $(\lambda,e^{i\theta}) \in \Lambda_\Sigma^*$). Assume, equivalently, that $\det(e^{-i\theta/2}F + e^{i\theta/2}F^* - e^{-i\theta/2}\lambda I) = 0$. Without loss of generality, assume that the self adjoint matrix $e^{-i\theta/2}F + e^{i\theta/2}F^*$ is diagonal. Then if $f_{11}, \ldots, f_{nn}$ denote the diagonal elements of $F$, then for some $i$ we must have
$$e^{-i\theta/2}f_{ii} + e^{i\theta/2}\overline{f_{ii}} = e^{-i\theta/2}\lambda .$$
This gives
$$ \overline{\lambda} e^{i\theta} = f_{ii} + e^{i\theta}\overline{f_{ii}} = \lambda .$$
By (4) of Theorem \ref{Gamma-isometry} we find that $(\lambda,e^{i\theta})$ is a $\Gamma$-isometry, so it is in $\Gamma$.

Putting everything together we obtain
\begin{align*}
\|f(S,P)\| \leq \max_{(s,p)
\in \Lambda^*_{\Sigma} \cap b \Gamma}\|f^\cup(s,p)\| =
\max_{(s,p) \in \Lambda_{\Sigma}\cap b\Gamma}\|f(s,p)\|.
\end{align*}
Finally, when $\omega(F)<1$, the set $\Lambda_{\Sigma}\cap \mathbb
G$ is a distinguished variety by Theorem \ref{thm:DVchar}.
\end{proof}

\begin{cor}\label{cor:VN}
Let $(S,P)$ be a strict $\Gamma$-contraction on a finite dimensional Hilbert space $\mathcal H$.
Then there is a distinguished variety $W \subset \mathbb G$ such
that for every matrix valued polynomial $f$,
$$
\|f(S,P)\| \leq \max_{(s,p) \in \partial W} \|f(s,p)\|.
$$
\end{cor}
\begin{proof}
This follows from the theorem above combined with Lemma
\ref{lem:strict} and Proposition \ref{prop:nrstrict}.
\end{proof}

\begin{thm}\label{last:thm}
Let $\Sigma = (S,P)$ be a $\Gamma$-contraction on a finite dimensional
Hilbert space $\mathcal H$ and let $F$ be the fundamental operator
of $(S,P)$. Then for every matrix valued polynomial $f$,
\begin{equation*}\label{eq:vnlam}
\|f(S,P)\| \leq \max_{(s,p) \in \Lambda_{\Sigma}\cap b\Gamma} \|f(s,p)\|,
\end{equation*}
where $\Lambda_{\Sigma}$ is the following one dimensional variety:
\begin{equation}\label{eq:lambda}
\Lambda_{\Sigma} = \{(s,p) \in \Gamma : \det(F^*+pF - sI) = 0\}.
\end{equation}
\end{thm}

\begin{proof}
Let $\{r_n\}$ be a sequence of positive numbers converging to $1$
from below. Then $\Sigma_n = (S_n,P_n) := (r_nS,r_n^2P)$ is a sequence of
$\Gamma$-contractions satisfying the conditions of Theorem
\ref{thm:VN}. Let $F_n$ denote the fundamental operator of
$\Sigma_n$, and let
\[
\Lambda_{\Sigma_n} = \{(s,p) \in \Gamma : \det(F_n^* + pF_n - sI)
= 0\}.
\]
We consider all $F_n$ as operators on the finite dimensional space
$\mathcal H$, and by passing to a subsequence, we may assume that
the sequence $F_n$ converges to some operator $F$. We proceed to
show that $F$ is the fundamental operator of $\Sigma$.

Since $F_n$ is the fundamental operator of $\Sigma_n$, we have
$$S_n-S_n^*P_n=(I-P_n^*P_n)^{\frac{1}{2}}F_n(I-P_n^*P_n)^{\frac{1}{2}}. $$
Taking the limit as $n\rightarrow \infty$, we get
$$S-S^*P=D_P F D_P.$$
Thus, by the uniqueness of the fundamental
operator, $F$ is indeed the fundamental operator of $\Sigma$.

For a matrix valued polynomial $f$, we have
$$
\|f(S,P)\| = \lim_n \|f(S_n,P_n) \| \leq \limsup_n \max_{(s,p) \in
\Lambda_{\Sigma_n}\cap b\Gamma} \|f(s,p)\|,
$$
where the inequality follows by applying Theorem
\ref{thm:VN} to $\Sigma_n$. It is
now easy to see that
\[
\|f(S,P)\| \leq \limsup_n \max_{(s,p) \in \Lambda_{\Sigma_n}\cap
b\Gamma} \|f(s,p)\| \leq \max_{(s,p) \in \Lambda_{\Sigma}\cap
b\Gamma} \|f(s,p)\| =\|f\|_{\Lambda_{\Sigma}}.
\]

\end{proof}

In Theorem \ref{thm:VN}, we have constructed a
$\Gamma$-co-isometric extension to the $\Gamma$-contraction
$(S,P)$ and the $\Gamma$-co-isometric extension is
$(M_{\varphi},M_z)$ defined on $H^2(\mathbb D) \otimes \mathcal
D_P$, which lives on $\Lambda_{\Sigma}$. By Proposition
\ref{easyprop3}, the adjoint of a $\Gamma$-co-isometric extension
of $(S,P)$ is a $\Gamma$-isometric dilation of $(S^*,P^*)$. So in
particular when $(S,P)$ is a $\Gamma$-contraction on a finite
dimensional Hilbert space, it has a $\Gamma$-isometric dilation
that lives on a one dimensional variety in $\Gamma$.

\end{document}